\documentclass[12pt]{amsart}
\usepackage{amsmath,amsthm,amsfonts,amssymb,latexsym,enumerate,url,hyperref,color} 

\usepackage[top=30mm,right=30mm,bottom=30mm,left=30mm]{geometry}

\begin{document}

\theoremstyle{plain}

\newtheorem{thm}{Theorem}[section]
\newtheorem{lem}[thm]{Lemma}
\newtheorem{conj}[thm]{Conjecture}
\newtheorem{pro}[thm]{Proposition}
\newtheorem{hyp}[thm]{Hypothesis}

\newtheorem{cor}[thm]{Corollary}
\newtheorem{que}[thm]{Question}
\newtheorem{rem}[thm]{Remark}
\newtheorem{defi}[thm]{Definition}

\newtheorem*{thmA}{THEOREM A}
\newtheorem*{thmB}{THEOREM B}
\newtheorem*{thmC}{THEOREM C}
\newtheorem*{conjC}{CONJECTURE C}
\newtheorem*{conjD}{CONJECTURE D}
\newtheorem*{thmD}{THEOREM D}
\newtheorem*{thmE}{THEOREM E}
\newtheorem*{conjF}{CONJECTURE F}

\newtheorem*{thmAcl}{Main Theorem$^{*}$}
\newtheorem*{thmBcl}{Theorem B$^{*}$}

\numberwithin{equation}{section}

\newcommand{\Maxn}{\operatorname{Max_{\textbf{N}}}}
\newcommand{\Syl}{\operatorname{Syl}}
\newcommand{\dl}{\operatorname{\mathfrak{d}}}
\newcommand{\Con}{\operatorname{Con}}
\newcommand{\cl}{\operatorname{cl}}
\newcommand{\Stab}{\operatorname{Stab}}
\newcommand{\Aut}{\operatorname{Aut}}
\newcommand{\Ker}{\operatorname{Ker}}
\newcommand{\IBr}{\operatorname{IBr}}
\newcommand{\Irr}{\operatorname{Irr}}
\newcommand{\SL}{\operatorname{SL}}
\newcommand{\FF}{\mathbb{F}}
\newcommand{\NN}{\mathbb{N}}
\newcommand{\N}{\mathbf{N}}
\newcommand{\C}{\mathbf{C}}
\newcommand{\OO}{\mathbf{O}}
\newcommand{\F}{\mathbf{F}}

\renewcommand{\labelenumi}{\upshape (\roman{enumi})}

\newcommand{\GL}{\operatorname{GL}}
\newcommand{\Sp}{\operatorname{Sp}}
\newcommand{\PGL}{\operatorname{PGL}}
\newcommand{\PSL}{\operatorname{PSL}}
\newcommand{\SU}{\operatorname{SU}}
\newcommand{\PSU}{\operatorname{PSU}}
\newcommand{\PSp}{\operatorname{PSp}}
\newcommand{\SO}{\operatorname{SO}}
\newcommand{\X}{\mathcal{X}}
\def\Inndiag{\mathop{{\rm Inndiag}}}

\providecommand{\V}{\mathrm{V}}
\providecommand{\E}{\mathrm{E}}
\providecommand{\ir}{\mathrm{Irr_{rv}}}
\providecommand{\Irrr}{\mathrm{Irr_{rv}}}
\providecommand{\re}{\mathrm{Re}}

\def\irrp#1{{\rm Irr}_{p'}(#1)}

\def\Z{{\mathbb Z}}
\def\C{{\mathbb C}}
\def\Q{{\mathbb Q}}
\def\irr#1{{\rm Irr}(#1)}
\def\ibr#1{{\rm IBr}(#1)}
\def\irrv#1{{\rm Irr}_{\rm rv}(#1)}

\def \c#1{{\cal #1}}
\def\cent#1#2{{\bf C}_{#1}(#2)}
\def\syl#1#2{{\rm Syl}_#1(#2)}
\def\nor{\unlhd}
\def\oh#1#2{{\bf O}_{#1}(#2)}
\def\Oh#1#2{{\bf O}^{#1}(#2)}
\def\zent#1{{\bf Z}(#1)}
\def\det#1{{\rm det}(#1)}
\def\ker#1{{\rm ker}(#1)}
\def\norm#1#2{{\bf N}_{#1}(#2)}
\def\alt#1{{\rm Alt}(#1)}
\def\iitem#1{\goodbreak\par\noindent{\bf #1}}
   \def \mod#1{\, {\rm mod} \, #1 \, }
\def\sbs{\subseteq}

\def\gc{{\bf GC}}
\def\ngc{{non-{\bf GC}}}
\def\ngcs{{non-{\bf GC}$^*$}}
\newcommand{\notd}{{\!\not{|}}}
\newcommand{\Out}{{\mathrm {Out}}}
\newcommand{\Mult}{{\mathrm {Mult}}}
\newcommand{\Inn}{{\mathrm {Inn}}}
\newcommand{\IBR}{{\mathrm {IBr}}}
\newcommand{\IBRL}{{\mathrm {IBr}}_{\ell}}
\newcommand{\IBRP}{{\mathrm {IBr}}_{p}}
\newcommand{\ord}{{\mathrm {ord}}}
\def\id{\mathop{\mathrm{ id}}\nolimits}
\renewcommand{\Im}{{\mathrm {Im}}}
\newcommand{\Ind}{{\mathrm {Ind}}}
\newcommand{\diag}{{\mathrm {diag}}}
\newcommand{\soc}{{\mathrm {soc}}}
\newcommand{\End}{{\mathrm {End}}}

\newcommand{\sol}{{\mathrm {sol}}}
\newcommand{\Hom}{{\mathrm {Hom}}}
\newcommand{\Mor}{{\mathrm {Mor}}}
\newcommand{\St}{{\sf {St}}}
\def\rank{\mathop{\mathrm{ rank}}\nolimits}
\newcommand{\Tr}{{\mathrm {Tr}}}
\newcommand{\tr}{{\mathrm {tr}}}
\newcommand{\Gal}{{\it Gal}}
\newcommand{\Spec}{{\mathrm {Spec}}}
\newcommand{\ad}{{\mathrm {ad}}}
\newcommand{\Sym}{{\mathrm {Sym}}}
\newcommand{\Char}{{\mathrm {char}}}
\newcommand{\pr}{{\mathrm {pr}}}
\newcommand{\rad}{{\mathrm {rad}}}
\newcommand{\abel}{{\mathrm {abel}}}
\newcommand{\codim}{{\mathrm {codim}}}
\newcommand{\ind}{{\mathrm {ind}}}
\newcommand{\Res}{{\mathrm {Res}}}
\newcommand{\Ann}{{\mathrm {Ann}}}
\newcommand{\Ext}{{\mathrm {Ext}}}
\newcommand{\Alt}{{\mathrm {Alt}}}
\newcommand{\AAA}{{\sf A}}
\newcommand{\SSS}{{\sf S}}
\newcommand{\CC}{{\mathbb C}}
\newcommand{\CB}{{\mathbf C}}
\newcommand{\RR}{{\mathbb R}}
\newcommand{\QQ}{{\mathbb Q}}
\newcommand{\ZZ}{{\mathbb Z}}
\newcommand{\KK}{{\mathbb K}}
\newcommand{\NB}{{\mathbf N}}
\newcommand{\ZB}{{\mathbf Z}}
\newcommand{\OB}{{\mathbf O}}
\newcommand{\EE}{{\mathbb E}}
\newcommand{\PP}{{\mathbb P}}
\newcommand{\GC}{{\mathcal G}}
\newcommand{\HC}{{\mathcal H}}
\newcommand{\AC}{{\mathcal A}}
\newcommand{\BC}{{\mathcal B}}
\newcommand{\GA}{{\mathfrak G}}
\newcommand{\SC}{{\mathcal S}}
\newcommand{\TC}{{\mathcal T}}
\newcommand{\DC}{{\mathcal D}}
\newcommand{\LC}{{\mathcal L}}
\newcommand{\RC}{{\mathcal R}}
\newcommand{\CL}{{\mathcal C}}
\newcommand{\EC}{{\mathcal E}}
\newcommand{\GCD}{\GC^{*}}
\newcommand{\TCD}{\TC^{*}}
\newcommand{\FD}{F^{*}}
\newcommand{\GD}{G^{*}}
\newcommand{\HD}{H^{*}}
\newcommand{\hG}{\hat{G}}
\newcommand{\hP}{\hat{P}}
\newcommand{\hQ}{\hat{Q}}
\newcommand{\hR}{\hat{R}}
\newcommand{\GCF}{\GC^{F}}
\newcommand{\TCF}{\TC^{F}}
\newcommand{\PCF}{\PC^{F}}
\newcommand{\GCDF}{(\GC^{*})^{F^{*}}}
\newcommand{\RGTT}{R^{\GC}_{\TC}(\theta)}
\newcommand{\RGTA}{R^{\GC}_{\TC}(1)}
\newcommand{\Om}{\Omega}
\newcommand{\eps}{\epsilon}
\newcommand{\varep}{\varepsilon}
\newcommand{\al}{\alpha}
\newcommand{\chis}{\chi_{s}}
\newcommand{\sigmad}{\sigma^{*}}
\newcommand{\PA}{\boldsymbol{\alpha}}
\newcommand{\gam}{\gamma}
\newcommand{\lam}{\lambda}
\newcommand{\la}{\langle}
\newcommand{\ra}{\rangle}
\newcommand{\hs}{\hat{s}}
\newcommand{\htt}{\hat{t}}
\newcommand{\sgn}{\mathsf{sgn}}
\newcommand{\SR}{^*R}
\newcommand{\tn}{\hspace{0.5mm}^{t}\hspace*{-0.2mm}}
\newcommand{\ta}{\hspace{0.5mm}^{2}\hspace*{-0.2mm}}
\newcommand{\tb}{\hspace{0.5mm}^{3}\hspace*{-0.2mm}}
\def\skipa{\vspace{-1.5mm} & \vspace{-1.5mm} & \vspace{-1.5mm}\\}
\newcommand{\tw}[1]{{}^#1\!}
\renewcommand{\mod}{\bmod \,}
\newcommand{\edit}[1]{{\color{red} #1}}

\newcommand{\carolinacomment}{\textcolor{cyan}}

\marginparsep-0.5cm

\renewcommand{\thefootnote}{\fnsymbol{footnote}}
\footnotesep6.5pt

\title{Principal Blocks, Irreducible Restriction, Fields and Degrees}

\author[R. Lyons]{Richard Lyons}
\address{Department of Mathematics, Rutgers University, Piscataway, NJ 08854, USA}
\email{lyons@math.rutgers.edu}

\author[J. M. Mart\'inez]{J. Miquel Mart\'inez}
\address{Departament de Matem\`atiques, Universitat de Val\`encia, 46100 Burjassot,
Val\`encia, Spain}
\email{josep.m.martinez@uv.es}

\author[G. Navarro]{Gabriel Navarro}
\address{Departament de Matem\`atiques, Universitat de Val\`encia, 46100 Burjassot,
Val\`encia, Spain}
\email{gabriel@uv.es}

\author[P. H. Tiep]{Pham Huu Tiep}
\address{Department of Mathematics, Rutgers University, Piscataway, NJ 08854, USA}
\email{tiep@math.rutgers.edu}
 
\thanks{The research of the second and third authors is supported by Grant PID2022-137612NB-I00
 funded by MCIN/AEI/ 10.13039/501100011033 and ERDF ``A way of making Europe."
 The fourth author gratefully acknowledges the support of the NSF (grant
 DMS-2200850) and the Joshua Barlaz Chair in Mathematics. Part of this work started while the third author visited Rutgers University.
He would like to thank the Department for the hospitality and support. The second author acknowledges grant CIDEIG/2022/29 from Generalitat Valenciana for funding a visit to the Universitat de Val\`encia where part of this work was done. }

\thanks{The authors thank J. Lynd for pointing out \cite{Gl}. The second author wishes to thank Mandi Schaeffer Fry for helpful discussions on the topic of this note.}
 
 \dedicatory{To I. M. Isaacs, in memoriam}
 
%
 


\begin{abstract}
Several recent problems in the representation theory of finite groups require determining whether certain characters of almost simple groups belong to the principal block. Since the values of these characters are not yet known, we employ alternative group-theoretical techniques to address the ``going down" case. This approach enables us to reduce the block version of well-known results by the third and fourth authors to a question about almost simple groups. Moreover, this suggests a Galois analogue of the height-zero-equal-degree conjecture of Malle and Navarro, which we formulate. However, the ``going up" case of irreducible extensions of principal block characters remains unresolved.
\end{abstract}

\maketitle

\section{Introduction}  The resolution of several recent conjectures concerning the characters of principal blocks of finite groups depends on a better understanding of when the irreducible restriction of a character in the principal block remains in the principal block, and conversely, particularly in the case of almost simple groups (see, for instance, \cite{LWXZ}, \cite{NRS}, or \cite{NST}).
Somewhat surprisingly, the character values of almost simple groups are not sufficiently well understood, which means that determining whether a character lies in the principal block often requires alternative techniques.

Let $p$ be a fixed prime, let $G$ be a finite group, and denote by $\irr{B_0(G)}$ the irreducible
complex characters of $G$ in the $p$-principal block.
For example, a useful going down theorem due to T. Okuyama states that if $\chi \in \Irr(B_0(G))$, the restriction $\chi_H$ is irreducible for a subgroup $H$ of $G$, and the induced block $b^G$ is defined (where $b$ is the block of $\chi_H$), then $\chi_H$ lies in the principal block of $H$ (see Theorem 6.6 of \cite{N1}). However, in practice, it is not always the case that induced blocks from subgroups are defined.

In the first result of this paper, we contribute to this problem in the following way.

 \begin{thmA}
 Let $G$ be an almost simple group with socle $S$. Let $S \le H \le G$.
 If $\chi \in \irr G$ is in the principal block of $G$ and $\chi_S$ is irreducible,
 then $\chi_H$ is in the principal block of $H$.
 \end{thmA} 
 
 Since, as we have noted, the character values of almost simple groups are not fully known, Theorem A relies on the following 
 deep group-theoretical result, which is not as well-known as it should be. Recall that when extending principal block characters from a normal subgroup $N$ of $G$ to $G$, the normal subgroup $V=N\cent GQ$, where
 $Q \in \syl pN$, plays a key role (by the Alperin--Dade theory on principal blocks).  Theorem B does use the classification of finite simple groups when $p$ is odd, but not when $p=2$.

 \begin{thmB}
 Suppose that $G$ is almost simple with socle $S$. Let $Q \in \syl pS$ and $V=S\cent GQ$.
Then $V/S$ has a normal $p$-complement.
\end{thmB}

\begin{proof}
If $p$ is odd, then  we have that $V/S$ is a $p'$-group by the main result of \cite{GGLN}.
If $p=2$, then this follows from Theorem 1 of \cite{Gl}.
\end{proof}

Theorem A is then obtained from the following elementary result.

\medskip

\begin{thmC}
Suppose that $N \nor G$. Let $Q\in \syl pN$ and $V=N\cent GQ$. Assume that $V/N$ has a normal $p$-complement.
 Suppose that $\chi \in \irr{B_0(G)}$
is such that $\chi_N=\theta$ is irreducible. If $N \le U \le G$, then  $\chi_U \in \irr{B_0(U)}$.
\end{thmC}

Theorem C is not necessarily true if $G/N$ does not have a normal $p$-complement, as shown by 
$G={\sf A}_4$, $p=2$, $U \in \syl 3G$, $N=1$, and $\chi=\lambda \in \irr G$ is linear and  not trivial.
%
\medskip

Why are we now interested in Theorem A? In \cite{NT3}, using fields of values, a powerful character-theoretic criterion for the existence of a normal $p$-complement was established, extending a classical result of Thompson. Specifically,  Theorem C of \cite{NT1}
 states that if $p$ divides $\chi(1)$ for every non-linear $\chi \in \mathrm{Irr}(G)$ whose values lie in the $p$-th cyclotomic field 
$$\Q_p=\Q(e^{2\pi i/p}),$$ 
then $G$ has a normal $p$-complement. For $p=2$, Theorem A of \cite{NT1} establishes both directions. It has long been suspected that it would be sufficient to impose this condition only on the characters of the principal $p$-block. By applying Theorem A, 
we can prove this now for $p = 2$; for odd primes $p$ 
we reduce the problem to   a question
of almost simple groups. The main difficulty that remains in that case is the absence of a convenient ``going up" version of Theorem A.
 
 \begin{thmD}
 Let $G$ be a finite group. Then $\chi(1)$ is even for all non-linear rational valued $\chi \in \irr{B_0(G)}$
 if and only if $G$ has a normal 2-complement.
 \end{thmD}

The $p$ odd case depends on whether simple groups satisfy the $(\star)$ condition \ref{con:simples}. We verify condition $(\star)$ for all primes for alternating groups and for groups of Lie type in their defining characteristic (see Propositions \ref{defi} and \ref{alt}) as well as for sporadic groups (this can be easily checked with \cite{GAP}).

\begin{thmE}
Let $p$ be any odd prime number and 
assume that every finite non-abelian simple group of order divisible by $p$ satisfies condition $(\star)$ of \ref{con:simples} for $(\mathbb{Q}_p,p)$. Then 
$p$ divides $\chi(1)$ for  all non-linear $\QQ_p$-valued $\chi \in \irr{B_0(G)}$ if and only if $G$
 has a normal $p$-complement.
\end{thmE}

Contrary to the case of Theorem C of \cite{NT1}, notice that using principal blocks,  we get an {\sl if and only if} statement in Theorem E.

\medskip

Could it be that Theorems D and E admit a version for general blocks? Some of the global/local conjectures and theorems on the representation theory of finite groups admit a Galois version. (See \cite{N2} for Alperin-McKay, 
and \cite{MN2} and \cite{MMRS} for Brauer's height zero on principal blocks, for instance.)  These provide interesting  
questions and sometimes unforeseen consequences, at the same time that they lead us 
to study the action of the absolute Galois group on key representation theory objects.
In \cite{MN1}, a global characterization of the Brou\'e--Puig nilpotent blocks \cite{BP} (which are locally defined objects) was proposed:
a Brauer block is nilpotent if and only if all of its height zero characters have the same degree.
To this day, this conjecture is wide open, even for solvable groups (where its proof depends on an unsolved regular orbit problem;
see \cite{G} for groups of odd order).  Inspired by Theorems D and E, the final purpose in this note is to propose a finer Galois version of this conjecture.  If $B$ is an arbitrary  $p$-block, it is not even true that there exists some irreducible character in $B$
having  values in $\Q(e^{2\pi i/p})$, so definitely something has to be changed.
We believe that the key characters in the general case are the so-called {\sl almost $p$-rational characters}, as defined in \cite{HMM}.
If $\chi \in \irr G$ and $c$ is its conductor (the smallest positive number such that the $c$-cyclotomic field $\Q_c$ contains the values of
$\chi$), recall that $\chi$ is almost $p$-rational if $p^2$ does not divide $c$.  It was proven by Brou\'e and Puig in \cite{BP} that  
every $p$-block $B$ contains at least one height zero almost $p$-rational character.
Extensive computations lead us to propose the following.

\begin{conjF}
A $p$-block $B$ is nilpotent if and only if all of its height zero almost $p$-rational characters in $B$ have the same degree.
\end{conjF}

Recall that the principal $p$-block is nilpotent if and only if $G$ has a normal
$p$-complement. Hence Conjecture F for the principal block says that a group has a normal $p$-complement if and only if all the almost $p$-rational characters in $B_0(G)$ are linear.  This is easily implied by Theorem E.

\section{Proof of Theorems A and C}

If $G$ is a finite group and $p$ is a fixed prime, let us denote by $B_0(G)$ the principal 
$p$-block of $G$.  Our notation follows \cite{N1} and \cite{N3}.  We begin by proving the elementary Theorem C.
We start with the following.

\begin{lem}\label{p'}
Suppose that $\chi \in \irr{B_0(G)}$, $U\le G$ has $p'$-index, and $\mu=\chi_U$ is irreducible.
Then $\mu \in \irr{B_0(U)}$.
\end{lem}

\begin{proof}
Let $x^U$ be a defect class for $\mu$, where $x \in U$.  (See the definition before Theorem 4.4 in \cite{N1}.)
Let $P$ be a Sylow $p$-subgroup of $\cent Ux$, so that $P$ is a defect group of the block of $\mu$.
In particular,
$$\left(\mu(x)|x^U| \over \mu(1) \right)^* \ne 0 \, .$$
By multiplying by the $p'$-number $|G:U|$, we obtain
$$0 \ne \left(\mu(x)|G:\cent Ux| \over \mu(1) \right)^*= \left(\chi(x)|G:\cent Gx| \over \chi(1) \right)^* |\cent Gx:\cent Ux|^*=|G:\cent Ux|^*\, .$$
We have that $P$ is contained in a Sylow $p$-subgroup of $\cent Gx$. By Problem 4.5 of \cite{N1}, we only need to prove that
$$\left(\chi(y) |y^U| \over \chi(1) \right)^*=|y^U|^*$$ for every $p$-regular $y \in U$ with $[y,P]=1$. 
We know that
$$\left(\chi(y) |y^G| \over \chi(1) \right)^*=|y^G|^* \, .$$
Notice that $${\chi(y)\over \chi(1)} = \left(\chi(y) |y^G| \over \chi(1) \right) /|y^G| \in S \, ,$$
where $S=\{\alpha/\beta | \alpha \in {\bf R}, \beta \in {\bf R}-M \}$, and $\bf R$ is the ring of algebraic integers in $\C$.
We deduce that 
$$\left( \chi(y)\over \chi(1) \right)^*=1$$
and the rest easily follows. 
\end{proof}

The following is Theorem C.

\begin{thm}
Suppose that $N \nor G$. Let $Q\in \syl pN$ and $V=N\cent GQ$. Assume that $V/N$ has a normal $p$-complement.
 Suppose that $\chi \in \irr{B_0(G)}$
is such that $\chi_N$ is irreducible. If $N \le U \le G$, then  $\chi_U \in \irr{B_0(U)}$.
\end{thm}

\begin{proof} We argue by induction on $|G:N|$ and then $|G:U|$. 
Suppose that $U<W<G$. Then we have that
  $\chi_W=\tau \in \irr{B_0(W)}$ by induction. We have that $\cent WQ N/N \le V/N$ has a normal $p$-complement.
By induction in $W$, we then have that $\tau_U \in \irr{B_0(U)}$. 
Hence, we may assume that $U$ is maximal in $G$.
Suppose that $V<G$. Notice that $\chi_V \in \irr{B_0(V)}$, because $V\nor G$. By induction,
we have that $\chi_{V\cap U}$ is in the principal block. Now, $V \cap U=N\cent UQ$, and thus any irreducible
character of $U$ above $\chi_{V\cap U}$ is in the principal block applying Lemma 3.1 of \cite{NT2}.
 So we may assume that
$V=G$. Let $K/N=\oh{p'}{G/N}$. Suppose that $K<G$.  If $KU=G$, then $U$ has $p'$-index and we are done
by Lemma \ref{p'}.
Otherwise,  $K \le U \le G$, $U\nor\nor G$, and we are done. 
\end{proof}

Now the proof of Theorem A easily follows from Theorems B and C.

\section{Proofs of Theorems D and E}

We start by making the following definition. 
\begin{defi}[Condition $(\star)$ for $(\mathbb{K},p)$]\label{con:simples}
Let $\mathbb{Q}\sbs\mathbb{K}$ be a field. Assume that $S$ is a nonabelian simple group of order divisible by $p$. Then we say $S$ sastisfies condition $(\star)$ for $(\mathbb{K},p)$ if there there is an $\Aut(S)$ orbit $\mathcal{X}\sbs\Irr(B_0(S))$ such that
\begin{enumerate}
\item $|\X|$ is not divisible by $p$,
\item the characters in $\X$ are nonlinear, of degree not divisible by $p$ and $\mathbb{K}$-valued,
\item every $\theta\in\X$ extends to a $\mathbb{K}$-valued character $\hat\theta$ in $B_0(\Aut(S)_\theta)$.
\end{enumerate}
\end{defi}

Thanks to Theorem A, part (iii) of Condition $(\star)$ implies that for any $S\leq T\leq \Aut(S)_\theta$ we can find a $\mathbb{K}$-valued character $\hat\theta\in\Irr(B_0(T))$ extending $\theta$.

\begin{pro}\label{pro:going to G}
Let $G$ be a finite group with a minimal normal subgroup $N=S_1\times\dots\times S_n $ where $S_i\cong S$ is a nonabelian simple group of order divisible by $p$ satisfying condition $(\star)$ for $(\mathbb{K},p)$, and $n$ is some integer. Then $G$ has a nonlinear $p'$-degree $\mathbb{K}$-valued character in the principal block.
\end{pro}
\begin{proof}
Let $\X$ be the orbit of $S$ that satisfies condition $(\star)$ and let $\X_i$ be the corresponding orbit for each $S_i$. We may assume $\cent G N=1$ because  $\cent G N\cap N=1$ and $\Irr(B_0(G/\cent G N))\sbs\Irr(B_0(G))$, so we may view $N\nor G\leq\Aut(N)\cong\Aut(S)\wr\mathsf{S}_n$. Define 
$$\mathcal{Y}=\{\alpha_1\times\dots\times\alpha_n\in\Irr(N)\mid\alpha_i\in \X_i\}$$
and notice that $\mathcal{Y}\sbs\Irr(B_0(N))$ by Lemma 2.6(ii) of \cite{NT3}.

Now $G$ acts on $\mathcal{Y}$ and since $|\mathcal{Y}|=|\mathcal{X}|^n$ is not divisible by $p$ there must be a $G$-orbit of size not divisible by $p$. Let $\lambda=\xi_1\times\dots\times\xi_n$ be a representative of this orbit, and set $K=\Aut(N)_\lambda\cong\Aut(S_1)_{\xi_1}\wr \mathsf{S}_n$. Notice that $K$ permutes the simple factors $S_1,\dots,S_n$ transitively.

We may view $S_1\cong S_1\cent K{S_1}/\cent K {S_1}\leq\norm{K}{S_1}/\cent K{S_1}\leq\Aut(S_1)_{\xi_1}$ because $K$ stabilizes $\lambda$. Let $\hat\xi_1$ be an extension of $\xi_1$ to $\Irr(B_0(\norm K{S_1}))$ with values in $\mathbb{K}$, which exists by Theorem A and Definition \ref{con:simples}(iii). By Corollary 10.5 of \cite{N3} (and its proof), the tensor induced character $\mu=(\hat\xi_1)^{\otimes K}$ extends $\lambda$. Moreover, by the tensor induction formula from Definition 2.1 of \cite{GI} we see that $\mathbb{Q}(\mu)\sbs\mathbb{Q}(\hat\xi_1)\sbs\mathbb{K}$.

Now let $$L=\bigcap_{i=1}^n\norm K{S_i}\nor K$$ and let $\{x_1,\dots, x_n\}$ be a $K$-transversal for $\norm K {S_1}$ so that $S_i=S_1^{x_i}$. We deduce from Lemma 10.4 of \cite{N3} that
$$\mu_L=(\hat\xi_1^{x_1})_L\cdots(\hat\xi_1^{x_n})_L$$
which is irreducible because $N\leq L$ and $\mu_N=\lambda$. 

Let $T=G_{\lambda}=G\cap K$. Our next goal is to show that $\mu_T\in\Irr(B_0(T))$.  Write $$M=\bigcap_{i=1}^t\norm T{S_i}=T\cap L\nor T$$ and notice that $N\leq M$. We have
$$S_1\cent K{S_1}/\cent K{S_1}\leq M\cent{K}{S_1}/\cent K{S_1}\leq L\cent{K}{S_1}/\cent{K}{S_1}\leq\norm K{S_1}/\cent{K}{S_1}$$
which implies by Theorem A that $(\hat\xi_1)_{M\cent{K}{S_1}}\in\Irr(B_0(M\cent{K}{S_1}/\cent{K}{S_1}))$. Now, via the natural isomorphism
$$M\cent{K}{S_1}/\cent{K}{S_1}\cong M/(M\cap \cent{K}{S_1})$$
we have that $(\hat\xi_1)_M\in\Irr(B_0(M))$ (using Problem 4.4 of \cite{N1}). Since $L\nor K$ we have $M^{x_i^{-1}}\leq L$ so $M^{x_i^{-1}}\leq \norm K {S_1}$, and arguing analogously we see that $(\hat\xi_1))_{M^{x_i^{-1}}}\in\Irr(B_0(M^{x_i^{-1}}))$ so $(\hat\xi_1^{x_i})_M\in\Irr(B_0(M))$. Then
$\mu_M=(\hat\xi_1^{x_1})_M\cdots(\hat\xi_1^{x_n})_M$ is irreducible (because $N\leq M$ and $\mu_N=\lambda$) and $\mu_M$
lies in $B_0(M)$ by Lemma 2.3 of \cite{GRSS}.

Let $Q\in\Syl_p(M)$ and $Q_i=Q\cap S_i\in\Syl_p(S_i)$. Then $$\cent T Q\sbs\bigcap_{i=1}^n{\cent T {Q_i}}\sbs\bigcap_{i=1}^n\norm T{S_i}=M.$$ We deduce from Lemma 3.1 of \cite{NT2} that $B_0(T)$ is the unique block covering $B_0(M)$. It follows that $\mu_T$ lies in $B_0(T)$, as desired.

By the Clifford correspondence, $\chi=(\mu_T)^G\in\Irr(G)$, and $\mathbb{Q}(\chi)\sbs\mathbb{Q}(\mu)\sbs\mathbb{K}$. Since $|G:T|$ is not divisible by $p$, $\chi$ has $p'$-degree. Finally, by Corollary 6.2 of \cite{N1}, $B_0(T)^G$ is defined and $\chi$ lies in $B_0(T)^G$ which is $B_0(G)$ by Brauer's third main theorem (Theorem 6.7 of \cite{N1}). This concludes the proof.
\end{proof}

\begin{hyp}\label{thm:simples}
Every nonabelian simple group of order divisible by $p$ satisfies condition $(\star)$ for $(\mathbb{Q}_p,p)$.
\end{hyp}

We denote by $\Irr_{p'}(B_0(G))$ the set of irreducible characters beloning to $B_0(G)$ whose degree is not divisible by $p$.

\begin{thm}\label{pro:psolvable}
Assume Hypothesis \ref{thm:simples}. Let $G$ be a finite group.
If every $\mathbb{Q}_p$-valued character $\chi\in\Irr_{p'}(B_0(G))$ is linear, then $G$ is $p$-solvable.
\end{thm}
\begin{proof}
We argue by induction on $|G|$. Let $N\nor G$ be a minimal normal subgroup. Since $\Irr(B_0(G/N))\sbs\Irr(B_0(G))$, by induction $G/N$ is $p$-solvable so we may assume that $N\cong S^t$ for some nonabelian simple group $S$ of order divisible by $p$. By   Proposition \ref{pro:going to G} there is a nonlinear $\Q_p$-valued $\psi\in\Irr_{p'}(B_0(G))$ and we are done.
\end{proof}

Of course, the ``if'' direction of Theorems D and E is trivial.
The following is the nontrivial direction of Theorem E.

\begin{cor}\label{only-if}
Let $p$ be any prime, and assume Hypothesis \ref{thm:simples}. Suppose that every $\mathbb{Q}_p$-valued character $\chi\in\Irr_{p'}(B_0(G))$ is linear. Then $G$ has a normal $p$-complement.
\end{cor}
\begin{proof}
We argue by induction on $|G|$. By Theorem \ref{pro:psolvable} we have that $G$ is $p$-solvable, so $\Irr(B_0(G))=\Irr(G/\oh{p'}G)$,
using \cite[Theorem 10.20]{N1}. Thus we may assume $\oh{p'}G=1$ by induction, and then $G$ has a unique block so we can apply \cite[Theorem 2.4]{NT3}.
\end{proof}

\begin{proof}[Proof of Theorem D]
By the preceding discussion and Corollary \ref{only-if}, it suffices to show that Hypothesis \ref{thm:simples} holds for $p=2$. Let $S$ be any finite non-abelian
simple group. By \cite[Theorem 3.3]{NT1}, $\Irr(S)$ contains a set $\mathcal{X}$ (of one or three characters) that satisfies all the conditions set in Hypothesis \ref{thm:simples}, except possibly for the requirement that $\hat\theta$ belongs to $B_0(\Aut(S)_\theta)$. But note that $\hat\theta$ is rational-valued and
of odd degree. On the other hand, by Fong's lemma, any nontrivial real-valued irreducible $2$-Brauer character of any finite group has even dimension.
It follows that the principal $2$-Brauer character must be a constituent of the restriction of $\hat\theta$ to $2'$-elements, and so 
$\hat\theta \in  B_0(\Aut(S)_\theta)$.
\end{proof}

Our next results prove Hypothesis \ref{thm:simples}, except for $S$ a simple Lie-type group in characteristic $\neq p$.

\begin{pro}\label{defi}
Let $p$ be any odd prime and let $S$ be a simple group of Lie type in characteristic $p$. Then $S$ satisfies condition $(\star)$ of \ref{con:simples}
for $(\QQ_p,p)$.
\end{pro}

\begin{proof}
By \cite[Theorem 3.3]{NT1}, $\Irr(S)$ contains a set $\mathcal{X}$ that satisfies all the requirements set in condition $(\star)$ of \ref{con:simples} 
for $(\QQ_p,p)$, except possibly for the one stipulating that $\hat\theta$ belongs to $B_0(\Aut(S)_\theta)$ for any $\theta \in \mathcal{X}$. Note that
any such $\theta$ is not the Steinberg character of $S$, hence it belongs to $B_0(S)$. By \cite[Lemma 6.3]{BLP},  $B_0(\Aut(S)_\theta)$ is the only 
$p$-block that covers $B_0(S)$, and so we are done.
\end{proof}

\begin{pro}\label{alt}
Let $p$ be any odd prime and let $S = \AAA_n$ with $n \geq 5$. Then $S$ satisfies condition $(\star)$ of \ref{con:simples}
for $(\QQ,p)$.
\end{pro}

\begin{proof}
First we assume $n \neq 6$, so that $A:=\Aut(S) \cong \SSS_n$. We will show that there is a character $\chi \in \Irr(A)$ such 
that $\mathcal{X} = \{\theta=\chi_S\}$ satisfies condition $(\star)$ of \ref{con:simples}
for $(\QQ,p)$. We will choose $\chi=\chi^\lambda$, the character of $A$ labeled by a suitable partition $\lambda \vdash n$, which has 
degree coprime to $p$ and larger than $1$, has the same $p$-core as of the partition $(n)$, and which is not self-associate.
The $p$-core condition implies that $\chi \in B_0(A)$, and the last condition guarantees that
$\chi$ restricts irreducibly to $S$, whence $\chi_S \in B_0(S)$ by Theorem A. Now, if $p|n$, then we choose $\lambda=(n-1,1)$, so that $\chi(1)=n-1$. 
If $n \equiv 1 \pmod{p}$, then we choose $\lambda=(n-2,2)$, so that $\chi(1)=n(n-3)/2$. Suppose now that $p \nmid n(n-1)$, whence
$n \equiv r \pmod{p}$ with $2 \leq r < p$. Let 
$$n= r + n_sp^s + n_{s+1}p^{s+1} + \ldots$$
be the $p$-adic decomposition of $n$, where $s \geq 1$ and $n_s \geq 1$. Then  
$$n-1= (r-1) + n_sp^s + n_{s+1}p^{s+1} + \ldots$$
is the $p$-adic decomposition of $n-1$; in particular, $p \nmid \binom{n-1}{p^s}$ by \cite[Lemma 22.4]{J}. Now we choose 
$\lambda$ to be the hook partition $(n-p^s,1^{p^s})$, so that $\chi(1) = \binom{n-1}{p^s}$, and note that $\lambda$ is not self-associate since
$p \nmid (n-1)$.

Assume now that $n=6$. If $p=3$, we can take $\theta$ to be the restriction to $S$ of the character $\chi^{(5,1)}$ of $\SSS_6$ as before (and note that
$\Aut(S)_\theta = \SSS_6$). If $p=5$, take $\theta$ to be of degree $9$, which has two extensions to $\Aut(S)$ that lie in the principal block.
\end{proof}

\section{Conjecture F}

If $\chi \in \irr G$, recall that $\Q(\chi)$ is the smallest subfield of the complex numbers containing the values $\chi(g)$ 
for $g \in G$, and that $c(\chi)$, the conductor of $\chi$,
 is the smallest positive integer $n$ such that $\Q(\chi)$ is contained in the $n$-th cyclotomic field $\Q_n$. If $n$ is an integer, then $n_p$ is the largest power of $p$ dividing $n$.
We say that $\chi \in \irr G$ is {\sl $p$-rational} if $c(\chi)_p=1$, and $\chi$ is {\sl almost $p$-rational} if $c(\chi)_p \le p$.  We start by providing an alternative proof of Proposition 2.6 of 
\cite{BP} from which we can extend a useful second part. (We thank C. Vallejo for pointing out this result to us.) Recall that the so called Galois--Alperin--McKay implies that every
$p$-block $B$ has an irreducible character  $\chi \in \irr B$ with $c(\chi)$ not divisible by $p$, but this fact has not been proved yet.

\medskip
\begin{thm}\label{p=2}
Let $B$ be a $p$-block of a finite group $G$. Then there exists  $\chi \in \irr B$ of height zero such that $c(\chi)_p\le p$.
Furthermore,  $\chi$ is unique if and only if $B$ has defect zero.
\end{thm}

\begin{proof}
Let $D$ be a defect group of $B$. Let $H=\norm GD$,
and by Brauer's First Main Theorem (4.14 of \cite{N1}), let $b$ be the unique block of $H$ with defect group $D$ inducing $B$. Let $C=\cent GD$  and let $\theta \in \irr{CD/D}$ be a canonical character of $b$
(see page 203 in \cite{N1}.)
Let $\tau \in \irr{H_\theta|\theta}$, where $H_\theta$ is the stabilizer
of $\theta$ in $H$. Recall that $H_\theta/DC$ is a $p'$-group by Theorem 9.22 of \cite{N1}. Then $\tau \in \irr{H_\theta/D}$ has
$p$-defect zero, and therefore, it is $p$-rational by Theorem 3.18 of \cite{N1}. Hence, $\xi=\tau^H \in \irr b$ is $p$-rational and has height zero since

$$\xi(1)_p=|H:H_\theta|_p\tau(1)_p=|H:H_\theta|_p|CD:D|_p=|H:D|_p.$$

By Corollary 6.4 of \cite{N1}, we have that
$$\left(\sum_{\chi \in \irr B}[\xi^G, \chi]\chi(1) \right)_p=\xi^G (1)_p=|G:H|_p|H:D|_p=|G:D|_p\, .$$
For $\chi \in \irr B$, write $\chi(1)=|G:D|_p p^{h_\chi} m_\chi$, where $m_\chi$ is not divisible by $p$, and $h_\chi \ge 0$.
Hence, we have that 
$$w= \sum_{\chi \in \irr B}[\xi^G, \chi]p^{h_\chi}m_\chi $$
is not divisible by $p$. Let $m=|G|_{p'}$, and let $A={\rm Gal}(\Q_{|G|}/\Q_{pm})$, a $p$-group. Since $A$ fixes 
the $p'$-roots of unity, notice that $A$ acts on $\irr B$.
Since $\chi^\sigma(1)=\chi(1)$ and $[\xi^G, \chi]=[\xi^G, \chi^\sigma]$ (because $\xi$ is $p$-rational), we have that
$$w=\sum_{i=1}^s |A:A_{\chi_i}|[\xi^G, \chi_i] p^{h_{\chi_i}}m_{\chi_i} \, ,$$
where $\chi_i$ are representatives of the action of $A$ on $\irr B$. 
We conclude that there exists $\chi_i \in \irr B$ of height zero which is $A$-invariant, 
and this proves the first part.

For the second part, let $e=\sum_{\chi \in \irr B} \chi(1)^2$ and write ${\rm Irr}_{0,A}(B)$ for the $A$-invariant height zero characters in $B$.
By Theorem 3.28 of \cite{N1}, we can write
$e=p^{2a-d}f$, where $p$ does not divide $f$, $|G|_p=p^a$ and $|D|=p^d$.  For each $\chi \in \irr B$, recall that
$\chi(1)=p^{a-d+h_\chi}m_\chi$, where $m_\chi$ is not divisible by $p$. Then
$$p^{2a-d} f=\sum_{\chi \in \irr B} p^{2h_\chi+ 2a -2d} m_\chi^2 \, .$$
Hence
$$p^d f= \sum_{\chi \in \irr B}p^{2h_\chi}m_\chi^2 \equiv \sum_{\chi \in {\rm Irr}_{0,A}(B)} m_\chi^2\,  \, {\rm mod} \, p \,.$$


If ${\rm Irr}_{0,A}(B)=\{\chi\}$, then we have that 
$$p^d f \equiv m_{\chi}^2\,  {\rm mod} \, p \,.$$
Since $m_\chi$ is not divisible by $p$, we obtain that $d=1$, as required.
\end{proof}

At least we can show Conjecture F for $p$-blocks with a normal defect group.

\begin{thm}
Let $G$ be a finite group, and let $B$ be a $p$-block of $G$ with defect group $D\nor G$.
If all the height zero almost $p$-rational characters in $B$ have the same degree, then $B$ is nilpotent.
\end{thm}

\begin{proof}
Let $N=\Phi(D)$. By Theorem 9.9 of \cite{N1}, $B$ contains a $p$-block $\bar B$ of $\bar G=G/N$
with defect group $D/N$. Let $\chi \in \irr{\bar B}$ and let $g \in \bar G$. If $g_p \not\in D/N$, then
$\chi(g)=0$, by the second main theorem. It follows that $\Q(\chi) \sbs \Q_{pm}$, where $m=|G|_{p'}$.
Hence, we have that all the irreducible characters of $\bar B$ are almost $p$-rational. It then follows that
$\bar B$ is nilpotent by Okuyama--Tsushima \cite{OT}. Now follow the proof of Theorem 5.2 of \cite{MN1}.
\end{proof}

\end{document}